\documentclass[a4paper,11pt]{amsart}
\usepackage{cases}

\usepackage{mathrsfs}
\usepackage{amsfonts,amssymb,amscd,amsthm,amsmath,graphicx}
\usepackage{longtable}
\usepackage[all]{xy}

\usepackage{graphicx}

\theoremstyle{plain}
\newtheorem{theorem}{Theorem}[section]
\newtheorem{lemma}[theorem]{Lemma}
\newtheorem{corollary}[theorem]{Corollary}
\newtheorem{proposition}[theorem]{Proposition}
\newtheorem{claim}[theorem]{Claim}
\newtheorem{definition}[theorem]{Definition}

\def\U{\operatorname{U}}

\def\Aut{\operatorname{Aut}}

\def\Im{\operatorname{Im}}
\def\Ind{\operatorname{Ind}}
\def\Int{\operatorname{Int}}

\def\max{\operatorname{max}}

\newcommand{\bbZ}{\mathbb{Z}}

\newcommand{\frg}{\mathfrak{g}}

\newcommand{\fru}{\mathfrak{u}}

\begin{document}

\title{A rigidity result for dimension data}

\author{Jun Yu}

\address{School of Mathematics, Institute for Advanced Study, Einstein Drive, Fuld Hall, Princeton, NJ 08540}
\email{junyu@math.ias.edu}

\abstract{The dimension datum of a closed subgroup of a compact Lie group is a sequence by assigning the invariant
dimension of each irreducible representation restricting to the subgroup. We prove that any sequence of dimension
data contains a converging sequence with limit the dimension datum of a subgroup interrelated to subgroups giving
this sequence. This rigidity has an immediate corollary that the space of dimension data of closed subgroups in a
given compact Lie group is sequentially compact. Moreover, we deduce some known finiteness and rigidity results
from it, as well as give an application to isospectral geometry.} %and continued in \cite{Andersen-Grodal-Moller-Viruel}
\endabstract

%\subjclass[2000]{aa}

%\thanks{}

\maketitle

\section{Introduction}

Given a compact Lie group $G$, let $\widehat{G}$ be the set of equivalence classes of irreducible
finite-dimensional complex linear representations of $G$. We endow the space \(\bbZ^{\widehat G}\)
with the product topology of the discrete topologies on the factors \(\bbZ\). Given a closed subgroup
$H$ of $G$, the function \[\mathscr{D}_H:\widehat{G}\to\bbZ, V \mapsto \dim V^{H} \] is an element
of \(\bbZ^{\widehat G}\), where $V^{H}$ is the linear space of $H$-invariant vectors in a complex
linear representation $V$ of $G$. We call $\mathscr{D}_H$ the {\it dimension datum} of \(H\).
In this way we have a map \[\mathscr{D}: \{\textrm{closed subgroup of }G\}/(G-\textrm{conjugacy})
\rightarrow\bbZ^{\widehat{G}},\ H\mapsto\mathscr{D}_{H}.\]
Dimension data arise from number theory in the determination of monodromy groups (cf. \cite{Katz}). They
also appear in differential geometry: in relation with the spectra of the  Laplace operators on Riemannian
homogeneous spaces (cf. \cite{Sutton}). Moreover, Langlands has suggested to use the dimension data as a key
ingredient in his program ``Beyond Endoscopy'' (cf. \cite{Langlands}, Page 2, Problem (II)). In the literature,
dimension data have been studied in the articles \cite{Larsen-Pink}, \cite{Larsen} and \cite{An-Yu-Yu}.

Endowing the space \(\bbZ^{\widehat G}\) (identified with the set of all maps from $\widehat{G}$ to $\bbZ$)
with the product topology of the discrete topologies on the factors \(\bbZ\), we could consider the
convergence of sequences of dimension data. Since $0\leq\dim V^{H}\leq \dim V$ for any complex irreducible
representation $V$ of $G$ and any closed subgroup $H$ of $G$, each sequence of dimension data
$\{\mathscr{D}_{H_{i}}:i\geq 1\}$ contains a convergent subsequence as elements in \(\bbZ^{\widehat G}\).
This raises a question: is the limit still a dimension datum $\mathscr{D}_{H}$ and does $\{H_{i}:i\geq 1\}$
give some restriction on $H$? Theorem \ref{T:rigidity} answers this question completely. It states that the
limit must be a dimension datum $\mathscr{D}_{H}$ and the subgroup $H$ is strongly controlled by a subsequence
$\{H_{i_{j}}:j\geq 1\}$. We also give some consequences of Theorem \ref{T:rigidity} by strengthening a theorem
of Larsen and give different proofs of two theorems proving in a joint paper \cite{An-Yu-Yu}, as well as giving
an application to isospectral geometry.

\begin{theorem}\label{T:rigidity}
Given a compact Lie group $G$, let $\{H_{n}|\ n\geq 1\}$ be a sequence of closed subgroups of $G$.
Then there exist two closed subgroups $H,H'$ of $G$ with $[H'_0,H'_0]\subset H\subset H'$ and
$H'$ quasi-primitive, a subsequence $\{H_{n_{j}}|\ j\geq 1\}$ and a sequence
$\{g_{j}|\ j\geq 1, g_{j}\in G\}$, such that
\[\forall j\geq 1, [H'_0,H'_0]\subset g_{j}H_{n_{j}}g_{j}^{-1}\subset H\] and
\[\lim_{j\rightarrow\infty}\mathscr{D}_{H_{n_{j}}}=\mathscr{D}_{H}.\]
\end{theorem}

This theorem has an immediate corollary.
\begin{corollary}\label{C:DT-closed}
The subset $\Im(\mathscr{D})\subset\bbZ^{\widehat{G}}$ is a closed subset of
$\bbZ^{\widehat{G}}$.
\end{corollary}

Corollary \ref{C:DT-closed} can be equivalently stated as follows.
\begin{corollary}\label{C:DT-compact}
The topological space $\Im(\mathscr{D})$ is sequentially compact.
\end{corollary}

Here the topology on $\Im(\mathscr{D})$ means the subspace topology inherited from the topology on
$\bbZ^{\widehat{G}}$.

Theorem \ref{T:rigidity} is equivalent to the combination of the following two theorems. Theorem
\ref{T:rigidity2} is a statement about dimension data, and Theorem \ref{T:group reduction2} is
an extension of Theorem \ref{T:group reduction} to non-finite closed subgroups.

\begin{theorem}\label{T:rigidity2}
Given a compact Lie group $G$, let $\{H_{n}|\ n\geq 1\}$ be a sequence of closed subgroups of $G$.
Then there exists a closed subgroup $H$ of $G$, a subsequence $\{H_{n_{j}}|\ j\geq 1\}$ and a
sequence $\{g_{j}|\ j\geq 1, g_{j}\in G\}$, such that
\[\forall j\geq 1, [H_0,H_0]\subset g_{j}H_{n_{j}}g_{j}^{-1}\subset H\] and
\[\lim_{j\rightarrow\infty}\mathscr{D}_{H_{n_{j}}}=\mathscr{D}_{H}.\]
\end{theorem}

\begin{theorem}\label{T:group reduction2}
Given a compact Lie group $G$ and a closed subgroup $H$, there exist a quasi-primitive subgroup
$H'$ of $G$ such that $[H'_0,H'_0]\subset H\subset H'$.
\end{theorem}

%Let $\mu_{H}$ be the unique normalized Haar measure of $H$ and $G^{\natural}$ be the set of conjugacy
%classes of elements of $G$. Let $\st_{H}$ be the push-forward of \(\mu_{H}\) by the composition
%\(H\hookrightarrow G\to G^\natural\), as a measure on \(G^\natural\). $\st_{H}$ is called the {\it
%Sato-Tate measure\/} of $H$. The dimension data $\mathscr{D}_{H}$ and the Sato-Tate measure $\st_{H}$
%carry the same information about the conjugacy class of the subgroup $H$ (\cite{Larsen-Pink}).

\noindent{\it Notation and conventions.}
Given a compact Lie group $G$,
\begin{enumerate}
\item let $G_0$ be the connected component of $G$ containing the identity element and
$[G,G]$ the commutator subgroup.
\item Write $G^{\natural}$ for the set of conjugacy classes in $G$.
\item Denote by $\mu_{G}$ the unique Haar measure on $G$ with $\int_{G} 1\mu_{G}=1$.
\item Write $\widehat{G}$ for the set of equivalence classes of irreducible finite-dimensional complex
linear representations of $G$.
\item Denote by $V^{G}$ the subspace of $G$-invariant vectors in a complex linear representation
$V$ of $G$.
\end{enumerate}

\noindent{\bf Acknowledgements.} The author thanks Jiu-Kang Yu, Jinpeng An, Brent Doran and Lars K\"uhne for
useful discussions and helpful suggestions. The influence of Michael Larsen's work in \cite{Larsen}
on this article is also obvious to the reader.

\section{Primitive and quasi-primitive subgroups}

In \cite{Yu}, we define primitive and quasi-primitive subgroups and prove a generalization of
a theorem of Borel and Serre by using them. In this section, we recall the definition and
some results in that paper. Lemma \ref{L:primitive closure}, Proposition \ref{P:primitive finite}
and Lemma \ref{L:group finite} are used in the proof of the main theorem of this paper. Lemma
\ref{L:group finite} is due to A. Weil (cf. \cite{}).

\begin{definition}\label{D:primitive}
A closed subgroup $H$ of a compact Lie group $G$ is called Lie primitive if $(A_{G})_{0}\subset H$
and for any closed subgroup $K$ of $G$ containing $H$ with $K_0$ normalizing $H_0$, either $[G:K]$ is
finite or $[K:H]$ is finite. It is called Lie quasi-primitive if there exists a sequence of closed subgroups
of $G$, $G=H_{0}\supset H_{1}\supset...\supset H_{s}=H$, such that each $H_{i+1}$ is
Lie primitive in $H_{i}$, $0\leq i\leq s-1$.
\end{definition}

\begin{theorem} \label{T:group reduction}
Given a finite subgroup $S$ of a compact Lie group $G$, there exists a Lie quasi-primitive subgroup $H$
of $G$ containing $S$ and satisfying that:
\begin{itemize}
\item[(1)] {$H_{0}$ is a Cartan subgroup of $G$, or}
\item[(2)] {$H_0$ is not abelian and $S A_{H}/A_{H}$ is a Lie primitive subgroup of $H/A_{H}$.}
\end{itemize}
\end{theorem}

\begin{lemma} \label{L:primitive closure}
Given a compact semisimple Lie group $G$ and a non-primitive closed subgroup $S$, there exists a
primitive subgroup $H$ of $G$ containing $S$ and such that $\dim H < \dim G$, and $S_0$ is a proper
normal subgroup of $H_0$.
\end{lemma}

\begin{proposition} \label{P:primitive finite}
Given a compact Lie group $G$, there exist only finitely many conjugacy classes of Lie primitive
and Lie quasi-primitive subgroups of $G$.
\end{proposition}

\begin{lemma}\label{L:group finite}
Given a finite group $S$ and a compact semisimple Lie group $G$, the number of conjugacy classes
of homomorphisms from $S$ to $G$ is finite.
\end{lemma}

\section{Proof of the main theorem}

An element in \(\bbZ^{\widehat G}\) can be expressed as a sequence $\vec{n}=\{a_{\rho}|\ \rho\in
\widehat{G}\}$, or equivalently as a function $\vec{n}: \widehat{G}\longrightarrow \bbZ$,
\[\vec{n}(\rho)=a_{\rho}, \forall\rho\in\widehat{G}.\]

In the topological space $\bbZ^{\widehat G}$, a sequence $\{\vec{n}_{i}|\ i\geq 1\}$ converges to $\vec{n}$ if
and only if for any finite subset $\{\rho_{j}|\ 1\leq j\leq m\}$ of $\widehat{G}$, we have
\[\lim_{i\rightarrow\infty}\vec{n}_{i}(\rho_{j})=\vec{n}(\rho_{j})\] for each $j$, $1\leq j\leq m$.

Let $X_{G}=\{\vec{n}\in\bbZ^{\widehat{G}}|\ 0\leq\vec{n}(\rho)\leq\dim\rho,\forall\rho\in\widehat{G}\}$. Then
$\mathscr{D}_{H}\in X_{G}$ for each closed subgroup $H$ of $G$. Since $X_{G}$ is a closed compact
subset, we have the following proposition.

\begin{proposition}\label{P:compact}
Any sequence $\{H_{n}|\ n\geq 1\}$ of closed subgroups of $G$ contains a subsequence
$\{H_{n_{j}}|\ j\geq 1\}$ such that $\mathscr{D}_{H_{n_{j}}}$ converges to some element $n$ in $X_{G}$.
\end{proposition}

\begin{proof}[Proof of Theorem \ref{T:rigidity}]

We prove it by induction on $\dim G$.

{\it Step 1. Reduction to the case that $[G_0,G_0]\subset H_{n}$ for any $n$.}

Let $\frg_0$ be the Lie algebra of $G$ and $\fru_0=[\frg_0,\frg_0]$ the derived subalgebra of
$\frg_0$. Then $\fru_0$ is a compact semisimple Lie algebra. There is an adjoint homomorphism
\[\pi: G\longrightarrow\Aut(\fru_0).\] Write $G'=\Aut(\fru_0)$. It is a compact semisimple Lie
group of adjoint type.

Assume first there are infinitely many $n$ such that $\dim\pi(H_{n})\neq \dim G'$, by replacing
$\{H_{n}|i\geq 1\}$ by a subsequence if necessary, we may assume that $\dim\pi(H_{n})<\dim G'$
for any $n\geq 1$. By Lemma \ref{L:primitive closure}, each $\pi(H_{n})$ is contained in a primitive
subgroup of $G'$ with a lower dimension. As primitive subgroups of $G'$ have only finitely many
conjugacy classes (cf. Proposition \ref{P:primitive finite}), we may assume that all $\pi(H_{n})$
are contained in a primitive subgroup $G_1$ of $G'$ with $\dim G_1<\dim G'$. Write
$G_2=\pi^{-1}(G_1)$. Since $\pi(H_{n})\subset G_1$, we have $H_{n}\subset G_2$ for any $n$. We
finish the proof in this case by induction.

Now assume there are only finitely many $n$ such that $\dim\pi(H_{n})\neq \dim G'$, by replacing
$\{H_{n}|\ n\geq 1\}$ by a subsequence if necessary, we may assume that $\dim\pi(H_{n})=\dim G'$ for
any $n\geq 1$. For any $n$, from $\dim\pi(H_{n})=\dim G'$ we get
\[\dim \Int(\fru_0)=\dim G'=\dim\pi(H_n)=\dim\pi(H_n\cap G_0).\] Thus
$\pi (H_n\cap G_0)=\Int(\fru_0)$ since $\pi(H_n\cap G_0)\subset\Int(\fru_0)$ and $\Int(\fru_0)$ is
connected.

We have $G_0=Z_{G_0}[G_0,G_0]$, $G_0/Z_{G_0}\cong\Int(\fru_0)$ and $\pi: G_0\longrightarrow
\Int(\fru_0)$ is just the projection map $G_0\longrightarrow G_0/Z_{G_0}$. For any $n$, since
$\pi (H_n\cap G_0)=\Int(\fru_0)$, we have $(H_{n}\cap G_0)Z_{G_0}=G_0$. Thus \[[G_0,G_0]=
[(H_{n}\cap Z_{G_0})Z_{G_0},H_{n}\cap Z_{G_0})Z_{G_0}]=[H_{n}\cap Z_{G_0},H_{n}\cap Z_{G_0}]
\subset H_{n}.\]

%Therefore it reduces to the case $[G_0,G_0]\subset H_{i}$ for any $i$.
From now on, we assume that $[G_0,G_0]\subset H_{n}$ for any $n$.

\smallskip

{\it Step 2. Reduction to the case that $G_0$ is abelian.}

Given a $\rho\in\widehat{G}$, if $\rho|_{[G_0,G_0]}$ does not contain any non-zero trivial
subrepresentation, we have $$\mathscr{D}_{H_{n}}(\rho)=\dim V_{\rho}^{H_{n}}=0.$$ If $\rho|_{[G_0,G_0]}$
contains a non-zero trivial subrepresentation, $\rho$ is a subrepresentation of
\[\Ind_{[G_0,G_0]}^{G}(\bf{1}_{[G_0,G_0]}).\] Here $\bf{1}_{[G_0,G_0]}$ is the $1$-dimensional trivial
representation of $[G_0,G_0]$. The action of $[G_0,G_0]$ on $\rho$ is trivial because $[G_0,G_0]$ is a
normal subgroup of $G$. Hence $\rho$ factors through the homomorphism $G\longrightarrow G/[G_0,G_0]$.
We only need to consider dimension data of the subgroups $\{H_{n}/[G_0,G_0]|\ n\geq 1\}$ of
$G/[G_0,G_0]$. As $(G/[G_0,G_0])_0$ is abelian, we reduce it to the case that $G_0$ is abelian.

From now on, we assume that $G_0$ is abelian.

\smallskip

{\it Step 3. Reduction to the case that $\{H_{n}\cap G_0|\ n\geq 1\}$ are all equal to some
$Z'\subset G_0$ and $Z'$ is a normal subgroup of $G$.}

Set $N_{n}=H_{n}/(H_{n}\cap C_{G}(G_0))$. Then $N_{n}\subset G/C_{G}(G_0)$. Since $G/C_{G}(G_0)$ is
a finite group, there are only finitely many possibilities for $N_{n}$. We may assume that they are
all equal to some $N\subset G/C_{G}(G_0)$. The finite group $N$ acts on $G_0$ by conjugation.

Consider the sequence $\{H_{n}\cap G_0|\ n\geq 1\}$. The sequence of dimension data associated to it
contains a convergent subsequence by Proposition \ref{P:compact}. We may assume that the sequence of
dimension data itself converges. Denote by $L$ be the dual group of $G_0$, i.e. the set of isomorphism
classes of $1$-dimensional complex linear representations of $G_0$. Then $L$ is a lattice. Write
$L'\subset L$ for the subset of irreducible representations $\rho$ of $G_0$ such that
\[\lim_{n\rightarrow\infty}\dim \rho^{H_{n}\cap G_0}=1.\]

\begin{claim}\label{Claim}
 $L'$ is an $N$-stable sublattice.
\end{claim}

We defer the proof of Claim \ref{Claim} and finish the proof of Theorem \ref{T:rigidity} first. Let
$Z'\subset G_0$ be the intersection of the kernels of the linear characters in $L'$. By Claim \ref{Claim},
$L'$ is $N$-stable. Thus $Z'$ is also $N$-stable. Choose a basis $\{\rho_{j}| 1\leq j\leq m\}$ of $L'$.
For any $j$, there exists an $k_{j}$ such that $\dim\rho_{j}^{H_{n}\cap G_0}=1$ for any $n\geq k_{j}$.
Write \[k=\max\{k_{j}: 1\leq j\leq m\}.\] For any $n\geq k$ and any $j$ with $1\leq j\leq m$, we have
$\dim\rho_{j}^{H_{n}\cap G_0}=1$. That is to say, $H_{n}\cap G_0$ acts trivially on $\rho_{j}$. As
$\rho_{j}$ is a basis of $L'$, we have $H_{n}\cap G_0$ acts trivially on any $\rho\in L'$. This means
\[H_{n}\cap G_0\subset Z'\] for any $n\geq k$. Replacing $\{H_{n}\cap G_0\subset Z'|\ n\geq 1\}$
by a subsequence if necessary, we may assume that $H_{n}\cap G_0\subset Z'$ for any $n\geq 1$.
Since $Z'$ is $N$-stable and $N_{n}=N$, we have $H_{n}\subset N_{G}(Z')$. Replacing $G$ by
$N_{G}(Z')$, we may and do assume that $Z'$ is normal in $G$. Let $H'_{n}=H_{n}Z'$.

Given a $\sigma\in\widehat{G}$, if $\sigma|_{G_0}$ does not contain any simple irreducible
sub-representation isomorphic to an element of $L'$, then
\[\lim_{n\rightarrow\infty}\dim\rho^{H'_{n}}=\lim_{n\rightarrow\infty}\dim\rho^{H_{n}}=0.\]

If $\sigma|_{G_0}$ contains a simple irreducible subrepresentation isomorphic to an element
of $L'$, the action of $Z'$ on $\sigma$ must be trivial. Thus \[\lim_{n\rightarrow\infty}
\dim\rho^{H'_{n}}=\lim_{n\rightarrow\infty}\dim\rho^{H_{n}Z'}=\lim_{n\rightarrow\infty}
\dim\rho^{H_{n}}.\]

%(after replace $\{H_{i}| i\geq 1\}$ by some $\{g_{i}H_{i}g_{i}^{-1}| i\geq 1\}$).

By the above, \[\lim_{n\rightarrow\infty}\mathscr{D}_{H'_{n}}=\lim_{n\rightarrow\infty}
\mathscr{D}_{H_{n}}.\] Replacing $\{H_{n}|\ n\geq 1\}$ by $\{H'_{n}|n\geq 1\}$, it reduces to
the case $\{H_{n}\cap G_0|\ n\geq 1\}$ are all equal to some $Z'\subset G_0$ and $Z'$ is a normal
subgroup of $G$.

From now on, we assume that $G_0$ is abelian, there is a subgroup $Z'$ of $G_0$ that is
a normal subgroup of $G$, and $H_{n}\cap G_0=Z'$ for any $n\geq 1$.

\smallskip

{\it Step 4. Conclusion.}

Arguing similarly as in Step 2, for any $\rho\in\widehat{G}$, if $\rho|_{Z'}$ does not contain any
non-zero trivial subrepresentation, then $\mathscr{D}_{H_{n}}(\rho)=\dim V_{\rho}^{H_{n}}=0$ for any
$n\geq 1$. If $\rho|_{Z'}$ contains a non-zero trivial subrepresentation, the action of $Z'$ on
$\rho$ is trivial. Hence $\rho$ factors through the homomorphism $G\longrightarrow G/Z'$. We only
need to consider the dimension data of the subgroups $\{H_{n}/Z'\}$ of $G/Z'$.

Now each $H_{n}/Z'$ is a finite group with order bounded by $|G/G_0|$. By Proposition
\ref{L:group finite}, there is a subsequence $\{H_{n_{j}}|\ j\geq 1\}$ of $\{H_{n}|\ n\geq 1\}$ such
that the subgroups in $\{H_{n_{j}}\}$ are conjugate to each other. Let $H$ be one of $\{H_{n_{j}}\}$
and $H'=G$. Then they satisfy the desired conclusion of the theorem. Going through the reduction steps,
we see that the subgroup $H'$ so constructed is a quasi-primitive subgroup of the original compact
Lie group $G$. \end{proof}

\begin{proof}[Proof of Claim \ref{Claim}]
Given a $\rho\in L'$, there exists a $k_{\rho}\geq 1$ such that $\dim V_{\rho}^{H_{n}\cap G_0}=1$
for any $n\geq k_{\rho}$. As $G_0$ is abelian, we have $\dim\rho=1$. The above equality means that
$H_{n}\cap G_0$ acts trivially on $\rho$ for any $n\geq k_{\rho}$. For another $\rho'\in L'$,
there also exists a $k_{\rho'}\geq 1$ such that $H_{n}\cap G_0$ acts trivially on $\rho'$
for any $n\geq k_{\rho'}$. Set $k=\max\{k_{\rho},k_{\rho'}\}$. Then $H_{n}\cap G_0$ acts trivially
on $\rho^{\pm{a}}\otimes\rho'^{\pm{1}}$ for any $n\geq k$. Here $\rho^{-1}$ is the contragredient
representation of $\rho$. Hence we have $\rho^{\pm{1}}\otimes\rho'^{\pm{1}}\in L'$. Therefore, $L'$ is
a sublattice of $L$.

For any $n\geq 1$, since $N_{n}=H/(H\cap C_{G_0})=N$, we get that $H_{n}$ is stable under the
conjugation action of $N$. Hence the set of those $\rho\in L$ such that
$\dim V_{\rho}^{H_{n}\cap G_0}=1$ is also stable under $N$. Therefore, $L'$ is stable under $N$.
\end{proof}

\begin{proof}[Theorem \ref{T:rigidity} implies Corollary \ref{C:DT-closed}]
By Theorem \ref{T:rigidity}, there exists a closed subgroup $H$ of $G$ and a subsequence
$\{H_{n_{j}}|\ j\geq 1\}$ such that \[\lim_{j\rightarrow\infty}\mathscr{D}_{H_{n_{j}}}=
\mathscr{D}_{H}.\] Hence \[\vec{n}=\lim_{n\rightarrow\infty}\mathscr{D}_{H_{n}}=
\lim_{j\rightarrow\infty}\mathscr{D}_{H_{n_{j}}}=\mathscr{D}_{H}.\]
\end{proof}

\section{Some consequences}\label{S:consequences}
Now we fix a compact Lie group $G$.

In \cite{Larsen}, given a finite-dimensional complex linear representation $V$ of $G$, Larsen studied
the moments \[\dim (V^{\otimes a}\otimes(V^{\ast})^{\otimes b})^{G}\] and got some rigidity properties
for them. Without loss of generality we may assume that $V$ is a unitary representation of $G$, it is
clear that the moments encode a part of the information of the dimension datum of the subgroup $G$ of
$\U(V)$. Theorem \ref{T3} is an analoguous statement of Theorem 4.2 in \cite{Larsen}. It is clear that
Theorem \ref{T3} implies Theorem 4.2 in \cite{Larsen}. We also use Theorem \ref{T:rigidity} to prove
some theorems in a joint paper \cite{An-Yu-Yu}. The proofs here are different and independent with
those in \cite{Larsen} and \cite{An-Yu-Yu}. Moreover, we give an application of Theorem
\ref{T:rigidity2} to isospectral geometry.

\begin{definition}\label{D:quasi-torus}
A compact Lie group $H$ is called a quasi-torus if $H_0$ is a torus.
\end{definition}

In the case that $\{H_{n}|\ n\geq 1\}$ in Theorem \ref{C:DT-closed} are all quasi-tori, we can say more about
the closed subgroup $H$. Theorem \ref{T3} below strengthens \cite{Larsen}, Theorem 4.2 by making the
convergence valid for dimension data rather than only for moments.

\begin{theorem}\label{T3}
If $\vec{n}\in\bbZ^{\widehat{G}}$ is the limit of a sequence of dimension data of quasi-tori of $G$, then
$\vec{n}=\mathscr{D}_{H}$ for some quasi-torus $H$ contained in $G$.
\end{theorem}

\begin{proof}
Denote by \[\vec{n}=\lim_{n\rightarrow\infty}\mathscr{D}_{H_{n}}.\]

By Theorem \ref{T:rigidity}, there exist two closed subgroups $H,H'\subset G$ with $[H'_0,H'_0]
\subset H\subset H'$, a subsequence $\{H_{n_{j}}|\ j\geq 1\}$ and a sequence
$\{g_{j}|\ j\geq 1, g_{j}\in G\}$ such that \[\forall j\geq 1, [H'_0,H'_0]\subset g_{j}H_{n_{j}}g_{j}^{-1}
\subset H\] and \[\lim_{j\rightarrow\infty}\mathscr{D}_{H_{n_{j}}}=\mathscr{D}_{H}.\]

As each $H_{n_{j}}$ is a quasi-torus and $[H'_0,H'_0]\subset g_{j}H_{n_{j}}g_{j}^{-1}$, we have
$[H'_0,H'_0]=0$. In other words, $H'_0$ is a torus. Therefore $H_0$ is a torus since $H\subset H'$. We also
have \[\vec{n}=\lim_{n\rightarrow\infty}\mathscr{D}_{H_{n}}=\lim_{j\rightarrow\infty}\mathscr{D}_{H_{n_{j}}}=
\mathscr{D}_{H}.\]
\end{proof}

%\begin{theorem}[\cite{Larsen} Theorem 4.2]\label{T3:Larsen}
%If $\mu$ is a Sato-Tate measure which is the limit of a sequence of Sato-Tate measures of
%representations of finite groups, then $\mu=\mu_{G,V}$ for some $G$ such that $G_0$ is a torus.
%\end{theorem}

\begin{lemma}\label{L:compare}
Given two closed subgroups $H$ and $H'$ of $G$, if $H\subset H'$ and $\mathscr{D}_{H}=\mathscr{D}_{H'}$,
then $H=H'$.
\end{lemma}

\begin{proof}
Suppose $H\subset H'$ is a proper inclusion. All $H'$-invariant functions in the space
$\Ind_{H}^{H'}\mathbf{1}_{H}=L^{2}(H'/H)$ are constant functions, so there exists a non-trivial
irreducible representation $\sigma$ of $H'$ appearing in $\Ind_{H}^{H'}\mathbf{1}_{H}$. Choosing
any irreducible representation $\rho$ of $G$ appearing in $\Ind_{H'}^{G}\sigma$, then
$\sigma\subset\rho|_{H'}$. Thus \[\dim\rho^{H}-\dim\rho^{H'}\geq\dim\sigma^{H}-
\dim\sigma^{H'}\geq 1-0>0.\] Hence $\mathscr{D}_{H}\neq\mathscr{D}_{H'}$.
\end{proof}

%Now we derive two theorems in \cite{An-Yu-Yu} from Theorem \ref{T:rigidity}.

\begin{theorem}\label{T:finite}
Given an element $\vec{n}\in\bbZ^{\widehat{G}}$, there are only finitely many conjugacy classes of closed
subgroups \(H\) of $G$ satisfying $\mathscr{D}_{H}=\vec{n}$.
\end{theorem}

\begin{proof}
We prove it by contradiction. Suppose there are infinitely many conjugacy classes of closed subgroups
\(H\) such that $\mathscr{D}_{H}=\vec{n}$. Then there exists a sequence of pairwise non-conjugate closed
subgroups $\{H_{n}|\ n\geq 1\}$ of $G$ such that $\mathscr{D}_{H_{n}}=\vec{n}$ for each $n\geq 1$.

By Theorem \ref{T:rigidity}, there exists a closed subgroup $H$ of $G$, a subsequence
$\{H_{n_{j}}|\ j\geq 1\}$ and a sequence $\{g_{j}|\ j\geq 1, g_{j}\in G\}$ such that
\[\forall j\geq 1, g_{j}H_{n_{j}}g_{j}^{-1}\subset H\] and
\[\lim_{j\rightarrow\infty}\mathscr{D}_{H_{n_{j}}}=\mathscr{D}_{H}.\]
As subgroups in the sequence $\{H_{n_{j}}|\ j\geq 1\}$ are non-conjugate to each other, we have
a proper inclusion $g_{j_0}H_{n_{j_0}}g_{j_0}^{-1}\subset H$  for $j_0=1$ or $2$. Moreover, we have
\[\mathscr{D}_{H}=\lim_{j\rightarrow\infty}\mathscr{D}_{H_{n_{j}}}=\vec{n}=
\mathscr{D}_{g_{j_0}H_{n_{j_0}}g_{j_0}^{-1}}.\] Write $H''=g_{j_0}H_{n_{j_0}}g_{j_0}^{-1}$. Then
$H''\subset H$ is a proper inclusion and $\mathscr{D}_{H''}=\mathscr{D}_{H}$. This is in contradiction
with Lemma \ref{L:compare}.
\end{proof}

%\begin{remark}
%In \cite{An-Yu-Yu}, proof of Theorem \ref{T:finite} is by some result about Lie group action
%and Lemma \ref{L:compare}, which is different the proof here.
%\end{remark}

\begin{theorem}\label{T:Larsen}
Given a closed subgroup $K$ of \(G\), put \(\vec{n}=\mathscr{D}_{K}\) and \(\{K_1,\ldots,K_r\}=
\mathscr{D}^{-1}(\vec{n})\). Then \(\vec{n}\) is an isolated point in \({\rm Im}(\mathscr{D})\) if
and only if \(K_1,\ldots,K_r\) are all semisimple groups.
\end{theorem}

\begin{proof}
We omit the proof for the ``only if'' part of the statement. There is a nice and short proof of this in
\cite{An-Yu-Yu}.

For the ``if'' part, suppose $n=\mathscr{D}_{K}$ is not an isolated point in $\Im(\mathscr{D})$.
Then there exists a sequence $\{H_{n}|\ n\geq 1\}$ of closed subgroups of $G$ such that
$\mathscr{D}_{H_{n}}\neq \vec{n}$ for any $n\geq 1$ and \[\lim_{n\rightarrow\infty}\mathscr{D}_{H_{n}}=\vec{n}.\]
By Theorem \ref{T:rigidity}, there exist two closed subgroups $H,H'\subset G$ with $[H'_0,H'_0]
\subset H\subset H'$, a subsequence $\{H_{n_{j}}|\ j\geq 1\}$ and a sequence
$\{g_{j}|\ j\geq 1, g_{j}\in G\}$ such that \[\forall j\geq 1, [H'_0,H'_0]
\subset g_{j}H_{n_{j}}g_{j}^{-1}\subset H\] and \[\lim_{j\rightarrow\infty}\mathscr{D}_{H_{n_{j}}}=
\mathscr{D}_{H}.\] From these we get $[H'_0,H'_0]\subset H_0\subset H'_0$ and $\vec{n}=\mathscr{D}_{H}$.
Thus $H\in\mathscr{D}^{-1}(\vec{n})$. As $\mathscr{D}^{-1}(\vec{n})$ consist of semisimple subgroups, we get
that $H$ is semisimple. Together with $[H'_0,H'_0]\subset H_0\subset H'_0$, we get $[H'_0,H'_0]=H_0$.
Hence $H_0\subset g_{j}H_{n_{j}}g_{j}^{-1}\subset H$ for any $j\geq 1$. Therefore there exists an
infinite set $n_{j_{\nu}}$ such that $g_{j_{\nu}}H_{n_{j_{\nu}}}g_{j_{\nu}}^{-1}$ are equal to each other.
We may and do assume that $\{H_{n_{j}}|\ j\geq 1\}$ are equal to each other. Write $H''$ for one of
them. Then $\mathscr{D}_{H''}=\mathscr{D}_{H}$. By Lemma \ref{L:compare} we get $H''=H$. This
is in contradiction with the fact that $\mathscr{D}_{H_{n}}\neq \vec{n}$ for any $n\geq 1$.
\end{proof}

%\begin{remark}\label{R:Larsen Thm 3.3}
%We can also prove Theorem 3.3 in \cite{Larsen} by Theorem \ref{T:rigidity}. The proof
%is along similar lines as the proof of Theorem \ref{T:Larsen}.
%\end{remark}

Given a closed Riemannian manifold $(X,m)$, one can define the Laplacian operator on the linear space of
$L^2$-functions on $X$. It is a self-adjoint positive definite linear operator, so its spectrum consists of
discrete real numbers, which is called the Laplacian Spectrum of $(X,m)$. Given a compact connected Lie group
$G$ with a biinvariant Riemannian metric $m$, for each closed subgroup $H$ of $G$ the homegeneous space
$G/H$ inherits a Riemannian metric, still denoted $m$. We are interested on the Laplacian spectra of these
homogeneous spaces here.

\begin{proposition}\label{P:Isospectral}
Given a a compact connected Lie group $G$ with a biinvariant Riemannian metric $m$, there exist finitely
many conjugacy classes of closed subgroups $H$ of $G$ with the Laplacian spectra of $(G/H,m)$ equal to a
given one.
\end{proposition}
\begin{proof}
We prove it by contradiction. Given a Laplacian spectrum, suppose there are infinitely many conjugacy classes
of closed subgroups \(H\) such that the spectra of $(G/H,m)$ equal to a given one. Then there exists a sequence
of pairwise non-conjugate closed subgroups $\{H_{n}|\ n\geq 1\}$ of $G$ such that the Laplacian spectra of
$(G/H,m)$ being equal to each other. By Theorem \ref{T:rigidity}, there exists a closed subgroup $H$ of $G$, a
subsequence $\{H_{n_{j}}|\ j\geq 1\}$ and a sequence $\{g_{j}|\ j\geq 1, g_{j}\in G\}$ such that
\[\forall j\geq 1, g_{j}H_{n_{j}}g_{j}^{-1}\subset H\] and
\[\lim_{j\rightarrow\infty}\mathscr{D}_{H_{n_{j}}}=\mathscr{D}_{H}.\] As subgroups in the sequence
$\{H_{n_{j}}|\ j\geq 1\}$ are non-conjugate to each other, we have a proper inclusion
$g_{j_0}H_{n_{j_0}}g_{j_0}^{-1}\subset H$  for $j_0=1$ or $2$. Write $H''=g_{j_0}H_{n_{j_0}}g_{j_0}^{-1}$.
Since the Laplacian spectrum of $(G/H,m)$ can be caculated from the dimension datum of $H$, the Laplacian
spectrum of $(G/H,m)$ is equal to  that of $H''$. Hence, $H''\subset H$ is a proper inclusion and
$(G/H,m)$ and $(G/H'',m)$ have the same Laplacian spectrum. This is in contradiction with Lemma \ref{L:compare}.
\end{proof}

\end{document}